\documentclass[12pt,a4paper]{amsart}
\usepackage{amsfonts}
\usepackage{amssymb}
\usepackage{amsmath}
\usepackage{amsthm}
\usepackage{cite}
\usepackage{enumitem}
\usepackage{tikz}

\usepackage{dsfont}

\theoremstyle{plain} 
\newtheorem{thm}{Theorem} 

\newtheorem{lem}{Lemma}
\newtheorem{prop}{Proposition}
\newtheorem{cor}{Corollary}
\newtheorem{rem}{Remark}

\providecommand{\skp}[2]{\langle#1,#2\rangle}

\providecommand{\sm}{\setminus}
\providecommand{\N}{\mathbb{N}}
\providecommand{\R}{\mathbb{R}}

\providecommand{\C}{\mathbb{C}}
\providecommand{\ind}{\mathds{1}}
\providecommand{\eps}{\varepsilon}
\providecommand{\ov}{\overline}

\providecommand{\skp}[2]{\langle#1,#2\rangle}
\providecommand{\les}{\lesssim}

\DeclareMathOperator{\supp}{supp}
\DeclareMathOperator{\sign}{sign}

\DeclareMathOperator{\Real}{Re}

\DeclareMathOperator{\rad}{rad}

\renewcommand{\qed}{\hfill $\Box$}

\renewcommand{\r}[1]{\textcolor{red}{#1}}
\renewcommand{\b}[1]{\textcolor{blue}{#1}}

\setitemize{itemsep=+2pt}
\setenumerate{itemsep=+2pt}

\begin{document}

\allowdisplaybreaks

\title[Dispersive estimates  for the Schr\"odinger equation]{Dispersive estimates, blow-up and failure of
Strichartz estimates for the Schr\"odinger equation with slowly decaying initial data}

\author{Rainer Mandel}
\address{R. Mandel \hfill\break
Karlsruhe Institute of Technology \hfill\break
Institute for Analysis \hfill\break
Englerstra{\ss}e 2 \hfill\break
D-76131 Karlsruhe, Germany}
\email{Rainer.Mandel@kit.edu}
\date{}

\subjclass[2000]{Primary: 35Q41, Secondary: 35B44, 35B40}
\keywords{Schr\"odinger equation, Failure of Strichartz estimates}

\begin{abstract}
  The initial value problem for the homogeneous Schr\"odinger equation is investigated for
  radially symmetric initial data with slow decay rates and not too wild oscillations.
  Our global wellposedness results apply to initial data for which
  Strichartz estimates fail.
\end{abstract}

\maketitle
\allowdisplaybreaks

\section{Introduction} 

In this paper we investigate the initial value problem for the Schr\"odinger equation     
\begin{equation}\label{eq:SchroedingerFlow}
  i\partial_t \psi + \Delta \psi = 0\quad\text{in }\R^n,\qquad \psi(0)=\phi
\end{equation}
for radial initial data $\phi$ with slow decay at infinity. In particular, we are interested in a solution
theory for~\eqref{eq:SchroedingerFlow} without assuming $\phi$ to belong to one of the Lebesgue spaces $L^r(\R^n)$
with $r\in [1,2]$. In this case Strichartz estimates are not available and local or
global well-posedness results for~\eqref{eq:SchroedingerFlow} are unknown. Surprisingly, we could not find a
complete statement about Strichartz estimates for such initial data in the literature, so we 
clarify this point here. 

 \begin{thm} \label{thm:NoStrichartz}
    Let $n\in\N,n\geq 2$ and $p,q\in [1,\infty], r>2$. Then there is no Strichartz estimate 
    \begin{equation}\label{eq:Strichartz_est}
      \|e^{it\Delta} \phi\|_{L^p_t(\R;L^q(\R^n))} \les \|\phi\|_{L^r(\R^n)}.
    \end{equation}
 \end{thm}
  
 This theorem partly generalizes the known fact that for any given $t>0$ the Schr\"odinger propagator
 $e^{it\Delta}$ is unbounded as a map from $L^r(\R^n)$ to $L^q(\R^n)$ for all $r>2,q\in [1,\infty]$,
 cf.~\cite[p.63]{LinPon_Dispersive} for the case $q=r$. Theorem~\ref{thm:NoStrichartz} may seem surprising in
 view of the fact that the optimal conditions for Strichartz estimates in the most important special case
 $r=2$ do not provide any obvious reason why the estimates should break down completely for $r>2$.
 Recall that these conditions are given by 
 $$
   p,q\geq 2, \qquad (p,q,n)\neq (2,\infty,2),\qquad \frac{2}{p}+\frac{n}{q}=\frac{n}{2},
 $$
 see for instance \cite[Theorem~2.3.3]{Caz_Semilinear}. We refer to the papers~\cite{Stri,KeelTao,GiVe} 
 for three milestone contributions related to the discovery of these conditions. At least for $n\geq 3$, each
 of the above conditions has a counterpart in the range $r>2$. The scaling invariance of the Schr\"odinger equation implies
 $\frac{2}{p}+\frac{n}{q}=\frac{n}{r}$ so that $q\geq r$ would be an immediate consequence that replaces
 the condition $q\geq 2$. As we discuss in the Appendix $p\geq 2$ generalizes to $p\geq
 \frac{2r}{(2n-r(n-1))_+}$. In particular, there is no evident reason for the necessity of $r\leq
 2$ so that Theorem~\ref{thm:NoStrichartz} seems to fill a gap in the literature.
 Its short proof relies on a thorough analysis of a counterexample due to Bona, Ponce, Saut
 and Sparber~\cite{BoPoSaSp_DispersiveBlowup}. The main feature of their solution is that the corresponding
 initial datum oscillates quadratically with respect to the distance to the origin, which produces blow-up of the solution at some
 prescribable finite time, cf. \cite[Lemma~2.1]{BoPoSaSp_DispersiveBlowup}. We reconsider this
 self-similar blow-up analysis for partly more general initial data and estimate the blow-up rate in
 $L^q(\R^n)$, which eventually leads to Theorem~\ref{thm:NoStrichartz}. 
 Accordingly, our proof even reveals that local Strichartz estimates cannot hold either and that no
 improvement in the radial situation is possible.
 
 \medskip
 
 Given that Strichartz estimates fail, the question arises how wellposedness results for the
 Schr\"odinger equation can be achieved if the initial datum lies in $L^r(\R^n)$ only for $r>2$. 
 In view of the above-mentioned counterexample it seems reasonable to impose a condition on the oscillations
 of the initial datum. In the following we present one possible approach in the radially symmetric case which
 relies on suitably weighted Sobolev norms of the initial data. For instance we identify a class of initial
 data lying in $L^r(\R^n)$ only for $r>\frac{2n}{n-1}$ with solutions that are bounded in time and uniformly
 localized in space, see Corollary~\ref{cor}. In that case dispersion need not occur because there are
 solutions of the form 
 \begin{equation} \label{eq:Herglotzwaves}
   \psi(x,t)=e^{-i\omega^2 t}\phi(x)
   \qquad\text{where }
   \phi(x)=|x|^{\frac{2-n}{2}}J_{\frac{n-2}{2}}(\omega |x|)
   \quad\text{for some }\omega\in\R.
 \end{equation} 
 Here, $J_{(n-2)/2}$ denotes the Bessel function of the first kind and $\phi$ solves the
 linear Helmholtz equation $\Delta \phi+\omega^2\phi=0$ in $\R^n$. Aiming for a more general result in this
 direction, we first consider initial profiles of the form $\phi(x)=e^{i\omega |x|}\phi_\omega(|x|)$ where
 $\phi_\omega$ belongs to the function spaces $X$ respectively $Y_m$ for some
 $m\in\{0,\ldots,n\}$ that we introduce now. The space $X$ is defined to be the completion of $C_c^\infty(\R_{\geq 0};\C)$ 
 with respect to the norm $\|\cdot\|_{X}:=\|\cdot\|_{X_1}+\|\cdot\|_{X_2}$
 given by
 \begin{align*}
   \|f\|_{X_1}
    &:=\sup_{z> 0} z^{\frac{1-n}{2}} \int_0^z \left(|f(r)| r^{n-2} + |f'(r)|  r^{n-1} \right)  \,dr, \\
    \|f\|_{X_2}
    &:= \int_0^\infty   \left|\frac{d}{dr}\left( f(r)r^{\frac{n-1}{2}}\right)\right| \,dr
         +\sup_{z> 0}   \left( z\int_z^\infty |f(r)|r^{\frac{n-5}{2}} \,dr  \right).
 \end{align*}
 Similarly, we define $Y_m$ to be the completion of  $C_c^\infty(\R_{\geq 0};\C)$ with respect to the norm
 \begin{align*}
      \|f\|_{Y_m} 
      &:= \sum_{k=0}^m \int_0^\infty |f^{(k)}(r)| r^{n-m+k-1}\,dr \qquad\qquad\text{if
      }m\in\{0,\ldots,n-1\},    \\
      \|f\|_{Y_n} 
      &:= \sum_{k=1}^n \int_0^\infty |f^{(k)}(r)| r^{k-1}\,dr + \sup_{z> 0} z^{-2}\int_0^z
      f(r)r\,dr + |f(0)|.
 \end{align*}
 One main feature of these spaces is that slow decay rates of its elements are admissible only provided that
 their derivatives decay fast enough.  For instance, we have that $r\mapsto (1+r)^{-\alpha}$ lies in $X$ if
 and only if $\alpha\geq \frac{n-1}{2}$ and in $Y_m$ if and only if $\alpha>n-m$
 whereas $r\mapsto e^{ir}(1+r)^{-\alpha}$ belongs to $X$ if and only if $\alpha>
 \frac{n+1}{2}$ and to $Y_m$ if and only if $\alpha>n$. 
 Using these spaces we find a local well-posedness theory for at most linearly oscillating radial initial
 data.
 
\begin{thm}\label{thm:dispest}
   Let $n\in\N, n\geq 2, m\in\{0,\ldots,n\}$ and $\phi(x)=\phi_\omega(|x|)e^{i\omega |x|}$ for some
   $\omega\in\R$.
  \begin{itemize}
    \item[(i)] If $\phi_\omega\in Y_m$  then~\eqref{eq:SchroedingerFlow} has a
    unique global mild solution $\psi$ satisfying
    \begin{align*}
      |\psi(x,t)| \leq C(\sqrt t)^{m-n} \|\phi_\omega\|_{Y_m}. 
    \end{align*}
   \item[(ii)] If $\phi_\omega\in X$  then \eqref{eq:SchroedingerFlow} has a unique global mild solution
    $\psi$ satisfying
    $$
      |\psi(x,t)| \leq C|x|^{\frac{1-n}{2}} \|\phi_\omega\|_X.
    $$
    \end{itemize}
    In (i) and (ii) the constant $C$ does not depend on $\omega$.
  \end{thm}


   Combining the estimates (i) and (ii) we deduce the following.
    
  \begin{cor} \label{cor}
     Let $n\in\N, n\geq 2$ and assume $\phi(x)= \int_\R  \phi_\omega(|x|)e^{i\omega |x|}\,d\mu(\omega)$ for
     some Borel measure $\mu$ on $\R$. Then~\eqref{eq:SchroedingerFlow} has a unique global mild solution
     satisfying 
     $$
       |\psi(x,t)| \leq C(1+|x|)^{\frac{1-n}{2}} \int_\R
       \big(\|\phi_\omega\|_X+\|\phi_\omega\|_{Y_n}\big)\,d\mu(\omega) 
     $$ 
     provided the right hand side is finite.
  \end{cor} 
  
  \begin{cor} \label{cor2}
     Let $n\in\N, n\geq 2,m\in\{0,\ldots,n\}$ and assume $\phi(x)= \int_\R  \phi_\omega(|x|)e^{i\omega
     |x|}\,d\mu(\omega)$ for some Borel measure $\mu$ on $\R$. Then~\eqref{eq:SchroedingerFlow} has a unique global mild solution
     satisfying 
     $$
       |\psi(x,t)| \leq C(1+t)^{-\frac{m}{2}} \int_\R
       \big(\|\phi_\omega\|_{Y_{n-m}}+\|\phi_\omega\|_{Y_n}\big)\,d\mu(\omega) 
     $$ 
     provided the right hand side is finite.
  \end{cor}

  \begin{rem} \label{rem} ~
     \begin{itemize}
     \item[(a)] Corollary~\ref{cor} applies to   superpositions of radially symmetric Herglotz
     waves given by some density $a\in L^1(\R)$. In fact, using the asymptotic expansions of the Bessel
     functions at infinity (see Proposition~\ref{prop:AnBnProperties} below) one can find functions $\eta_\omega \in X\cap
     Y_n$ such that
     \begin{align*}
       \phi(x)
       &:= \int_\R a(\omega) |x|^{\frac{2-n}{2}}J_{\frac{n-2}{2}}(\omega|x|)\,d\omega \\ 
       &= \int_\R a(\omega) \left( \eta_\omega(|x|) e^{i\omega |x|} + \ov{\eta_\omega(|x|)} e^{-i\omega
       |x|}\right)\,d\omega \\
       &= \int_\R \phi_\omega(|x|) e^{i\omega |x|} \,d\omega 
     \end{align*}
     where $\phi_\omega:= a(\omega)\eta_\omega+a(-\omega)\ov{\eta_\omega}\in X\cap Y_n$.
     In particular, Corollary~\ref{cor} generalizes the observation that the solutions of the
     Schr\"odinger equation with initial data given by superpositions of radially symmetric Herglotz waves
     as in~\eqref{eq:Herglotzwaves} remain bounded in time and uniformly localized in space.
     \item[(b)] Theorem~\ref{thm:dispest}~(i) is a generalized
     version of the fact that integrable initial data yield bounded solutions. Indeed, the latter statement 
     corresponds to $m=0$ in the theorem. So we get that less integrability of the initial datum is
     still sufficient for the absence of finite time blowup provided that the derivatives decay
     sufficiently fast. Notice that some kind of control on the
     derivatives seems necessary given that there are initial data in $L^s(\R^n)$ for any given $s>1$ the
     corresponding solutions of which blow up in finite time, see~\cite[Lemma~2.1 and
     Remark~2.2]{BoPoSaSp_DispersiveBlowup}.
    \item[(c)] The decay rates from Corollary~\ref{cor} improve
    once we add regularity assumptions on $\mu$. In the simplest situation $d\mu(\omega)= a(\omega)\,d\omega$ for 
    $a\in W_0^{k,1}(\R)$ and $\omega\mapsto \phi_\omega \in W^{k,\infty}(\R;X\cap Y_n)$, 
    the decay rate improves to $(1+|x|)^{\frac{1-n}{2}-k}$. This is proved using
    integration by parts as in the method of stationary phase. In the Appendix we discuss the densities
    $a(\omega)=(\omega-1)^{-\delta} \ind_{[1,2]}(\omega)$ with $\delta\in (0,1)$ and find the intermediate
    decay rates $(1+|x|)^{\frac{1-n}{2}-(1-\delta)}$. 
    \item[(d)] Theorem~\ref{thm:dispest}~(i) tells us that non-dispersive solutions of
    the Schr\"odinger equation~\eqref{eq:SchroedingerFlow} can only occur for radial
    initial data $\phi(x)=\phi_{\rad}(|x|)$ that satisfy $\|\phi_{\rad}\|_{Y_m}=\infty$ for all
    $m\in\{0,\ldots,n-1\}$. For smooth initial data this essentially means that for some
    $k\in\{0,\ldots,n-1\}$ the function $|\phi_{\rad}^{(k)}(r)|$ does not decay faster than $(1+r)^{-1-k}$ as
    $r\to\infty$.
    We conclude that the lack of dispersion is a phenomenon related to slowly decaying or heavily oscillating initial data.
  \end{itemize} 
 \end{rem}

%
%

\section{Proof of Theorem~\ref{thm:dispest}} \label{sec:DispEst}

In the following let $\psi$ denote the unique mild solution of the Schr\"odinger
equation~\eqref{eq:SchroedingerFlow} with initial datum $\phi$ given by
$\phi(x)=\phi_{\rad}(|x|)=\phi_\omega(|x|)e^{i\omega |x|}$. By density, il suffices to prove the estimates for
$\phi_\omega\in C_c^\infty(\R_{\geq 0};\C)$. In the following we use the abbreviations
\begin{equation}\label{eq:def_fg}
  f(r):= \phi_{\rad}(r) r^{\frac{n}{2}} J_{\frac{n-2}{2}}\Big(\frac{r|x|}{2t}\Big) \qquad\text{and}\qquad 
  g(\rho):= f(2\sqrt t \rho).
\end{equation}
We first recall the representation formula of the solution for radial initial data.
 
\begin{prop} \label{prop:SolutionFormulaI}
  We have for all $x\in\R^n,t>0$ 
  \begin{equation}\label{eq:SolutionFormulaI}
    \psi(x,t)
    =  |x|^{\frac{2-n}{2}}(\sqrt t)^{-1} e^{i(\frac{|x|^2}{4t}-\frac{n\pi}{4})} \int_0^\infty  g(\rho) e^{i\rho^2}\,d\rho.
  \end{equation}
\end{prop}
\begin{proof} 
  From (4.2) in~\cite{LinPon_Dispersive} or~(2.2.5) in~\cite{Caz_Semilinear} we get
  \begin{align*}
  \psi(x,t)
  &= \frac{1}{(4i\pi t)^{n/2}} \int_{\R^n} e^{i\frac{|x-y|^2}{4t}}\phi(y)\,dy \\
  &= \frac{1}{(4i\pi t)^{n/2}} \int_0^\infty \phi_{\rad}(r) r^{n-1} e^{i\frac{|x|^2+r^2}{4t}} 
    \left( \int_{\partial B_1(0)} e^{i\frac{r\skp{x}{\omega}}{2t}}\,d\sigma(\omega) \right) \,dr \\
  &= \frac{1}{(4i\pi t)^{n/2}} \int_0^\infty \phi_{\rad}(r) r^{n-1} e^{i\frac{|x|^2+r^2}{4t}} 
    \left( \int_{\partial B_1(0)} e^{i\frac{r|x|\omega_1}{2t}}\,d\sigma(\omega) \right) \,dr \\ 
  &= \frac{1}{(4i\pi t)^{n/2}} \int_0^\infty \phi_{\rad}(r) r^{n-1} e^{i\frac{|x|^2+r^2}{4t}} 
    \cdot |\partial B_1(0)|\Gamma(n/2)2^{\frac{n-2}{2}} \Big(\frac{r|x|}{2t}\Big)^{\frac{2-n}{2}} 
    J_{\frac{n-2}{2}}\Big(\frac{r|x|}{2t}\Big) \,dr \\
  &= |\partial B_1(0)|\Gamma(n/2)(4\pi i)^{-n/2}2^{n-2} 
   |x|^{\frac{2-n}{2}} t^{-1} e^{i\frac{|x|^2}{4t}}\int_0^\infty f(r) e^{i\frac{r^2}{4t}}  \,dr \\
  &= |\partial B_1(0)|\Gamma(n/2)(4\pi i)^{-n/2}2^{n-1}   |x|^{\frac{2-n}{2}} (\sqrt t)^{-1}
  e^{i\frac{|x|^2}{4t}} \int_0^\infty f(2\sqrt{t}\rho) e^{i\rho^2} \,d\rho. 
\end{align*}
  So the claim follows from $|\partial B_1(0)|=2\pi^{n/2}/\Gamma(n/2)$.
\end{proof}
 
 It will be convenient to split the
 integrand in~\eqref{eq:SolutionFormulaI} into three parts $g=g_1+g_2+g_3$.  
 The function $g_1$ will be identical to $g$  for small arguments and the corresponding estimates rely on
 the behaviour of the Bessel function $J_{(n-2)/2}$ on the interval $[0,1]$. The sum  $g_2+g_3$ represents $g$
  for large arguments and their definitions are based on the asymptotic expansion of the Bessel function
  at infinity. To be more precise, we fix
 some cut-off function $\chi\in C_0^\infty(\R)$ such that $\chi\equiv 1$ on $[0,\frac{1}{2}]$ and
 $\chi\equiv 0$ on $[1,\infty]$. Then, similar as in \cite[p.202]{Wat_Bessel}, we write 
 \begin{equation}\label{eq:splittingBessel}
   z^{n/2} J_{\frac{n-2}{2}}(z) = A_n(z) + e^{iz} B_n(z) + e^{-iz}\ov{B_n(z)}. 
 \end{equation}
 where the functions $A_n,B_n$ are given by  
 \begin{align}\label{eq:defnAnBn}
    A_n(z):= \chi(z) z^{n/2}J_{\frac{n-2}{2}}(z),\qquad  
    B_n(z):= (1-\chi(z))e^{-i \frac{(n-1)\pi}{4}} \sum_{k=0}^\infty \alpha_k z^{\frac{n-1}{2}-k}
 \end{align}
 and the  coefficients $\alpha_k\in\C$ are  $\alpha_0:=\frac{1}{\sqrt{2\pi}} $ and for $k\in\N$
 \begin{align}\label{eq:def_alphak}
   \begin{aligned}
   \alpha_k &:= \frac{1}{\sqrt{2\pi}} \left(\frac{n-2}{2},k\right) \left(\frac{i}{2}\right)^k
   \qquad\text{where } \\
   (\nu,k) &:= \frac{(4\nu^2-1^2)(4\nu^2-3^2)\cdot\ldots\cdot(4\nu^2-(2k-1)^2)}{4^k k!},
   \end{aligned}
 \end{align}
 see \cite[p.199]{Wat_Bessel}. The motivation for this decomposition is that $A_n,B_n$ and its derivatives
 satisfy useful uniform estimates that we provide next.
 
 \begin{prop} \label{prop:AnBnProperties}
  We have $\supp(A_n)\subset [0,1],\supp(B_n)\subset [\frac{1}{2},\infty)$ and for all $j\in\N,z\in\R$ 
  \begin{align*}
    |A_n^{(j)}(z)| 
    &\les \begin{cases}
      |z|^{n-1-j} &, \text{if }j\in\{0,\ldots,n-1\},\\
      |z| &, \text{if } j\in\{n,n+2,n+4,\ldots\}, \\
      1 &, \text{if }j\in\{n+1,n+3,\ldots\}, 
    \end{cases} \\
    \left|\frac{d^j}{dz^j} \left(B_n(z)z^{\frac{1-n}{2}}\right)\right|
    &\les \begin{cases}
      1 &, \text{if } j=0,  \\ 
      |z|^{-1-j} &, \text{if }j\geq 1. 
    \end{cases} 
  \end{align*}
\end{prop}
\begin{proof}
  The estimate for $A_n$ follows from
  \begin{align} \label{eq:An_series}
    \begin{aligned}
    A_n(z) 
    &= \chi(z)z^{\frac{n}{2}}J_{\frac{n-2}{2}}(z) \\
    &= \chi(z)z^{\frac{n}{2}} \sum_{m=0}^\infty \frac{(-1)^m}{m!\Gamma(m+\frac{n}{2})}
    \left(\frac{z}{2}\right)^{\frac{n-2}{2}+2m} \\
    &= \chi(z)  \sum_{m=0}^\infty \frac{(-1)^m}{2^{\frac{n-2}{2}+2m}m!\Gamma(m+\frac{n}{2})} z^{n-1+2m}.
    \end{aligned}
  \end{align}
  The estimate for $B_n$ follows from its series representation~\eqref{eq:defnAnBn}, see also
  \cite[p.206]{Wat_Bessel}.
\end{proof}
  
 Given the definition of $g$ in~\eqref{eq:def_fg} and the splitting~\eqref{eq:splittingBessel} of the Bessel
 function, we decompose the integrand $g$ according to $g=g_1+g_2+g_3$ where, for $r:=2\sqrt t\rho$,
  \begin{align*}
   \begin{aligned}
   g_1(\rho)
    &:= \phi_{\rad}(r)A_n\left(\frac{r|x|}{2t}\right)
    \left(\frac{|x|}{2t}\right)^{-\frac{n}{2}}
    = e^{2i\sqrt t\omega \rho }\phi_\omega(r)A_n\left(\frac{r|x|}{2t}\right)
    \left(\frac{|x|}{2t}\right)^{-\frac{n}{2}} ,\\
  g_2(\rho)
    &:=e^{i\tfrac{r|x|}{2t}}
    \phi_{\rad}(r)B_n\left(\frac{r|x|}{2t}\right) \left(\frac{|x|}{2t}\right)^{-\frac{n}{2}}
    =e^{2i\sqrt t (\omega+\tfrac{|x|}{2t})\rho}
    \phi_\omega(r)B_n\left(\frac{r|x|}{2t}\right) \left(\frac{|x|}{2t}\right)^{-\frac{n}{2}}  
    ,\\
  g_3(\rho)
    &:= e^{-i\tfrac{r|x|}{2t}} \phi_{\rad}(r) \ov{B_n\left(\frac{r|x|}{2t}\right)}
    \left(\frac{|x|}{2t}\right)^{-\frac{n}{2}} 
    = e^{2i\sqrt t(\omega-\tfrac{|x|}{2t})\rho} \phi_\omega(r) \ov{B_n\left(\frac{r|x|}{2t}\right)}
     \left(\frac{|x|}{2t}\right)^{-\frac{n}{2}}.
 \end{aligned}
 \end{align*}
 We now remove the linear phase factors by putting $g_{j,a_j}(\rho):=g_j(\rho)e^{-ia_j\rho}$ for  
  \begin{align}\label{eq:defa1a2a3}
    a_1 := 2\sqrt t \omega, \qquad a_2:=  2\sqrt t \left(\omega + \frac{|x|}{2t}\right), \qquad 
    a_3 :=  2\sqrt t  \left(\omega - \frac{|x|}{2t}\right).
  \end{align}
  This implies, again  for $r:=2\sqrt t\rho$,
  \begin{align}\label{eq:defg1a1g2a2g3a3}
   \begin{aligned}
   g_{1,a_1}(\rho)
    &= \phi_\omega(r)A_n\left(\frac{r|x|}{2t}\right) \left(\frac{|x|}{2t}\right)^{-\frac{n}{2}},\\
   g_{2,a_2}(\rho)
    &=  \phi_\omega(r)B_n\left(\frac{r|x|}{2t}\right) \left(\frac{|x|}{2t}\right)^{-\frac{n}{2}},\\
   g_{3,a_3}(\rho)
    &= \phi_\omega(r)\ov{B_n\left(\frac{r|x|}{2t}\right)}  \left(\frac{|x|}{2t}\right)^{-\frac{n}{2}}.
 \end{aligned}
 \end{align} 
 So we infer from Proposition~\ref{prop:SolutionFormulaI}
 \begin{align} \label{eq:represetation_psi}
    \psi(x,t) |x|^{\frac{n-2}{2}}\sqrt{t}   e^{-i(\frac{|x|^2}{4t}-\frac{n\pi}{4})}  
    = \int_0^\infty g(\rho)e^{i\rho^2}\,d\rho  
    = \sum_{j=1}^3 \int_0^\infty g_{j,a_j}(\rho)e^{i(\rho^2+a_j\rho)}\,d\rho.  
 \end{align}
 In order to estimate these terms, we make use of the following auxiliary result.
 
  \begin{prop}\label{prop:Xia}
    Let $a\in\R$ and $\Xi^m_a\in C^\infty(\R;\C)$ for $m\in\N_0$ be inductively defined by
  $$
    \Xi^0_a(s):=  \int_s^\infty e^{i(\rho^2+a\rho)}\,d\rho,\qquad
    \Xi^m_a(s):=   \int_s^\infty \Xi_a^{m-1}(\rho)\,d\rho.
  $$
  Then $|\Xi^m_a(s)|\leq C_m$ for all $a\in\R$,$s\geq 0$.  
 \end{prop}
 \begin{proof}
   This follows from $\Xi^m_a(s) = e^{-ia^2/4} \Xi^m_0(s+\frac{a}{2})$ once we have proved the estimate
   $$
     |\Xi^m(s)|\leq C_m(1+s_+)^{-m-1} \qquad\text{for all }s\in\R
   $$
   where $\Xi^m:=\Xi^m_0$. The existence of the improper Fresnel integral $\Xi^0(s)$ is a well-known
   consequence of the Residue Theorem. Moreover, l'H\^{o}pital's rule gives 
   $$ 
     \Xi^0(s)(-2is)e^{-is^2} \to 1 \qquad\text{and}\qquad 
     s^2(1+2is\Xi^0(s)e^{-is^2})\to \frac{i}{2} \quad\text{as } s\to\infty.
   $$ 
   (For the second limit one may proceed as we do below in the computation of $z_k$.)
   Since the improper integral $\int_s^\infty
   \frac{1}{\rho}e^{i\rho^2}\,d\rho$ exists for $s>1$ (again by the Residue Theorem), we obtain from the
   previous statement that the integral $\Xi^1(s)$ exists and 
   $$
     \Xi^1(s)(-2is)^2 e^{-is^2} \to 1 \qquad\text{and}\qquad  
     s^2(1-(-2is)^2\Xi^1(s)e^{-is^2})\to \frac{3i}{2} \quad\text{as } s\to\infty.
   $$ 
   By induction, we find that for all $k\in\N,k\geq 2$  the (proper) integral $\Xi^k(s)$ exists and
   \begin{align*}
     z_k&:=\lim_{s\to\infty} s^2\big(1-(-2is)^{k+1}\Xi^k(s)e^{-is^2}\big) \\
     &= \lim_{s\to\infty} \frac{ e^{is^2} (-2is)^{-k-1} - \Xi^k(s)}{  e^{is^2} s^{-2}(-2is)^{-k-1}} 
      \\
     &= \lim_{s\to\infty} \frac{ -e^{is^2} (-2is)^{-k}+2(k+1)i e^{is^2}(-2is)^{-k-2} 
      - (\Xi^k)'(s)}{
     - e^{is^2} s^{-2}(-2is)^{-k}
     - 2 e^{is^2} s^{-3}(-2is)^{-k-1}
     + 2(k+1)i e^{is^2} s^{-2}(-2is)^{-k-2} 
     } \\
     &= \lim_{s\to\infty} \frac{ - (-2is)^{-k}+2(k+1)i (-2is)^{-k-2} 
      + \Xi^{k-1}(s)e^{-is^2}}{
     -   s^{-2}(-2is)^{-k}
      - 2 s^{-3}(-2is)^{-k-1}
     + 2(k+1)i s^{-2}(-2is)^{-k-2}
     } \\
     &= \lim_{s\to\infty} \frac{ - (k+1)is^{-2} 
       - 2 + 2(-2is)^k \Xi^{k-1}(s)e^{-is^2}}{
     -   2s^{-2}  - 2is^{-4}  - (k+1)is^{-4}} \\
     &= \lim_{s\to\infty} \frac{ 
       (k+1)i + 2s^2( 1 - (-2is)^{k} \Xi^{k-1}(s)e^{-is^2})}{
         2   + 2is^{-2}  + (k+1)is^{-2}  
     } \\
     &=   \frac{(k+1)i}{2} + z_{k-1}
   \end{align*} 
   implying
   $$ 
     z_k = \lim_{s\to\infty} s^2\big(1-(-2is)^{k+1}\Xi^k(s)e^{-is^2}\big)= \frac{(k+2)(k+1)i}{4}.
   $$
   This yields the bounds for $\Xi^m(s)$ and we are done.  
 \end{proof}
  
 Let us remark that the proof actually yields the stronger estimate $|\Xi^m_a(s)|\leq C_m$ for
 $0\leq s\leq a_-$ and $|\Xi^m_a(s)|\leq 2^{m+1}C_m(1+s)^{-m-1}$ for $s\geq a_-$. However, given that
 these estimates depend on $a$, it seems difficult to make use of them. Moreover, the independence of $a$
 guarantees that our estimates below do not depend on $\omega$ since the latter is completely absorbed in the
 definition of $a_1,a_2,a_3$ from~\eqref{eq:defa1a2a3}. 
 From~\eqref{eq:represetation_psi} and Proposition~\ref{prop:Xia} we deduce the following estimate for the
 solution $\psi$.

\begin{prop}\label{prop:SolutionEstimate}
  For all $x\in\R^n,t\in\R$, $m\in\{0,\ldots,n\}$ we have
  \begin{align} \label{eq:SolutionEstimate}
   |\psi(x,t)| 
   \leq C_{m-1}  |x|^{\frac{2-n}{2}} t^{-\frac{1}{2}}   \left(
   \delta_{m,n} |g_{1,a_1}^{(n-1)}(0)| +
   \int_0^\infty |g_{1,a_1}^{(m)}(\rho)|+|g_{2,a_2}^{(m)}(\rho)|+|g_{3,a_3}^{(m)}(\rho)| \,d\rho  
   \right)    
  \end{align}
  where $g_{1,a_1},g_{2,a_2},g_{3,a_3}$ are given by \eqref{eq:defg1a1g2a2g3a3}. 
 \end{prop}
 \begin{proof}
  In the following integration-by-parts scheme we use $(\Xi^m_{a_j})'=-\Xi^{m-1}_{a_j}$ as well as
  \begin{equation} \label{eq:gjatzero}
    g_{j,a_j}(0)=g_{j,a_j}'(0)=\ldots=g_{j,a_j}^{(n-2)}(0)=0,\qquad 
    |g_{2,a_2}^{(n-1)}(0)|=|g_{3,a_3}^{(n-1)}(0)|=0,
  \end{equation}
  which follows from~\eqref{eq:defg1a1g2a2g3a3} and Proposition~\ref{prop:AnBnProperties}. 
  Recall that the support of $B_n$ is contained in $[\tfrac{1}{2},\infty)$ by choice of the cut-off function
  $\chi$ so that the above estimate  is actually trivial for $j\in\{2,3\}$. So we have  for
  $m\in\{0,\ldots,n-1\}$
  \begin{align} \label{eq:integrationbyparts}
    \begin{aligned}
    \int_0^\infty g_{j,a_j}(\rho)e^{i(\rho^2+a_j\rho)}\,d\rho  
    &\hspace{-1mm}\stackrel{\eqref{eq:gjatzero}}=  
    \lim_{M\to\infty} \int_0^\infty \left(\int_0^\rho g_{j,a_j}'(t)\,dt \right)
    e^{i(\rho^2+a_j\rho)}\,d\rho  \\
    &=  \lim_{M\to\infty} \int_0^M  g_{j,a_j}'(t) \left(\int_t^M
    e^{i(\rho^2+a_j\rho)}\,d\rho  \right) \,dt \\
    &=  \int_0^\infty g_{j,a_j}'(t) \Xi_{a_j}^0 (t) \,dt \\
    &\hspace{-1mm}\stackrel{\eqref{eq:gjatzero}}=  \int_0^\infty \left(\int_0^t g_{j,a_j}''(s)\,ds\right)
    \Xi^0_{a_j}(t) \,dt \\
    &=  \int_0^\infty g_{j,a_j}''(s) \Xi^1_{a_j}(s) \,ds \\
    &=  \ldots \\
    &=  \int_0^\infty g_{j,a_j}^{(m)}(\rho) \Xi_{a_j}^{m-1}(\rho)\,d\rho.  
  \end{aligned}
  \end{align}
  Notice that the limit $M\to\infty$ passes under the integral because $g_{j,a_j}$ has compact support and
  the $\Xi^k_{a_j}$ are bounded by Proposition~\ref{prop:Xia}. So we obtain for
  $m\in\{0,\ldots,n-1\}$
  \begin{align*} 
    |\psi(x,t)|  
    &\hspace{-1mm}\stackrel{\eqref{eq:represetation_psi}}\leq |x|^{\frac{2-n}{2}}(\sqrt{t})^{-1} \sum_{j=1}^3
    \left| \int_0^\infty g_{j,a_j}^{(m)}(\rho) \Xi^{m-1}_{a_j}(\rho)\,d\rho \right| \\
    &\leq C_{m-1} |x|^{\frac{2-n}{2}}(\sqrt{t})^{-1} \sum_{j=1}^3   \int_0^\infty
    |g_{j,a_j}^{(m)}(\rho)|  \,d\rho. 
  \end{align*}
  Moreover, using~\eqref{eq:integrationbyparts} for $m=n-1$ we get  
  \begin{align*}
    |\psi(x,t)|   
     &= |x|^{\frac{2-n}{2}}(\sqrt{t})^{-1}  \left|\sum_{j=1}^3 \int_0^\infty g_{j,a_j}^{(n-1)}(\rho)
     \Xi_{a_j}^{n-2}(\rho)\,d\rho \right|\\
     &= |x|^{\frac{2-n}{2}}(\sqrt{t})^{-1}\left|\sum_{j=1}^3 \left( g_{j,a_j}^{(n-1)}(0) \int_0^\infty
     \Xi_{a_j}^{n-2}(\rho)\,d\rho + \int_0^\infty \left(\int_0^t g_{j,a_j}^{(n)}(\rho)\,d\rho\right) \Xi_{a_j}^{n-2}(t)\,dt \right) \right|\\  
    &= |x|^{\frac{2-n}{2}}(\sqrt{t})^{-1} \left|\sum_{j=1}^3 \left( g_{j,a_j}^{(n-1)}(0) \Xi_{a_j}^{n-1}(0)  
    +  \int_0^\infty  g_{j,a_j}^{(n)}(\rho) \Xi^{n-1}_{a_j}(\rho)\,d\rho \right)\right| \\
    &\hspace{-1mm}\stackrel{\eqref{eq:gjatzero}}\leq  
      C_{n-1} |x|^{\frac{2-n}{2}}(\sqrt{t})^{-1} \left(|g_{1,a_1}^{(n-1)}(0)| +  \sum_{j=1}^3  \int_0^\infty 
      |g_{j,a_j}^{(n)}(\rho)| \,d\rho \right).
\end{align*}
\end{proof}
 
  For notational convenience we write $x\les y$ respectively $x\gtrsim y$ instead of
  $x\leq c\,y$ respectively $x\geq c\,y$ for positive numbers $c$ that are independent
  of $\omega,|x|,t,r$ but may depend on $m\in\{0,\ldots,n\}$ or the space dimension $n\in\N$.

  \begin{prop}\label{prop:estimategja}
   Let $m\in \{0,\ldots,n\}$. Then the functions $g_{1,a_1},g_{2,a_2},g_{3,a_3}$
   from~\eqref{eq:defg1a1g2a2g3a3} satisfy the following estimates for all $\rho\geq 0$ and $r=2\sqrt t \rho$:
    \begin{align*}
      |g_{1,a_1}^{(m)}(\rho)|
      &\les |x|^{\frac{n-2}{2}} (\sqrt t)^{m-n+2} \sum_{k=0}^m  
      |\phi_\omega^{(k)}(r)| r^{n-m+k-1}\cdot \ind_{[0,1]}\left(\frac{r|x|}{2t}\right) \qquad\text{if }m<n\\          
      |g_{1,a_1}^{(n)}(\rho)|
      &\les \left( |x|^{\frac{n-2}{2}} (\sqrt t)^2
        \sum_{k=1}^n  |\phi_\omega^{(k)}(r)| r^{k-1}
      +  |x|^{\frac{n+2}{2}} (\sqrt t)^{-2} |\phi_\omega(r)|r \right)  \cdot
      \ind_{[0,1]}\left(\frac{r|x|}{2t}\right)  \\
      |g_{2,a_2}^{(m)}(\rho)|+|g_{3,a_3}^{(m)}(\rho)|
      &\les |x|^{-\frac{1}{2}}(\sqrt t)^{1+m}  \sum_{k=0}^{m} 
      |\phi_\omega^{(k)}(r)|r^{\frac{n-1}{2}-m +k}
       \cdot   \ind_{[1/2,\infty)}\left(\frac{r|x|}{2t}\right).
    \end{align*} 
  \end{prop}
  \begin{proof}
    We get for $r=2\sqrt t \rho$ and $m\in\{0,\ldots,n-1\}$
    \begin{align*}
      |g_{1,a_1}^{(m)}(\rho)|
      &\stackrel{\eqref{eq:defg1a1g2a2g3a3}}{=} (2\sqrt t)^m \left|\frac{d^m}{dr^m}
      \left(\phi_\omega(r)A_n\left(\frac{r|x|}{2t}\right) \right)\right| 
      \left(\frac{|x|}{2t}\right)^{-\frac{n}{2}}    \\
      &\les (\sqrt t)^m \sum_{k=0}^m  \left| \phi_\omega^{(k)}(r) 
       A_n^{(m-k)}\left(\frac{r|x|}{2t}\right)\right| 
       \left(\frac{|x|}{2t}\right)^{m-k-\frac{n}{2}}  \\
      &\hspace{-3mm}\stackrel{\text{Prop. }\ref{prop:AnBnProperties}}{\les}  (\sqrt t)^m \sum_{k=0}^m  
      |\phi_\omega^{(k)}(r)| \left(\frac{r|x|}{2t}\right)^{n-1-m+k}
      \left(\frac{|x|}{2t}\right)^{m-k-\frac{n}{2}} \cdot \ind_{[0,1]}\left(\frac{r|x|}{2t}\right)   \\
      &\les |x|^{\frac{n-2}{2}} (\sqrt t)^{m-n+2} \sum_{k=0}^m  
      |\phi_\omega^{(k)}(r)| r^{n-1-m+k}\cdot \ind_{[0,1]}\left(\frac{r|x|}{2t}\right). \\          
    \intertext{This implies the first estimate. For $m=n$ we use the estimate for $A_n^{(j)}$ from 
    Proposition~\ref{prop:AnBnProperties} for $j\in\{0,\ldots,n\}$ and obtain
    }
      |g_{1,a_1}^{(n)}(\rho)|
      &\stackrel{\eqref{eq:defg1a1g2a2g3a3}}{=} (2\sqrt t)^n \left|\frac{d^n}{dr^n}
      \left(\phi_\omega(r)A_n\left(\frac{r|x|}{2t}\right) \right)\right| 
      \left(\frac{|x|}{2t}\right)^{-\frac{n}{2}}    \\
      &\les (\sqrt t)^n \sum_{k=0}^n  \left| \phi_\omega^{(k)}(r) 
       A_n^{(n-k)}\left(\frac{r|x|}{2t}\right)\right| 
       \left(\frac{|x|}{2t}\right)^{\frac{n}{2}-k}  \\
      &\les  (\sqrt t)^n \left( \sum_{k=1}^n  
      |\phi_\omega^{(k)}(r)| \left(\frac{r|x|}{2t}\right)^{-1+k} \left(\frac{|x|}{2t}\right)^{\frac{n}{2}-k}
      +   |\phi_\omega(r)| \frac{r|x|}{2t}  \left(\frac{|x|}{2t}\right)^{\frac{n}{2}} 
       \right) \cdot \ind_{[0,1]}\left(\frac{r|x|}{2t}\right)  \\
      &\les \left( |x|^{\frac{n-2}{2}} (\sqrt t)^2
        \sum_{k=1}^n |\phi_\omega^{(k)}(r)| r^{k-1}
      +  |x|^{\frac{n+2}{2}} (\sqrt t)^{-2} |\phi_\omega(r)|r \right)  \cdot
      \ind_{[0,1]}\left(\frac{r|x|}{2t}\right). 
   \end{align*}
   This yields the second estimate. The third estimate results from 
   \begin{equation} \label{eq:Bnestimates}
     |B_n^{(j)}(z)|\les  |z|^{\frac{n-1}{2}-j},
   \end{equation}
   which is a consequence of Proposition~\ref{prop:AnBnProperties}.
    \begin{align*}
      |g_{2,a_2}^{(m)}(\rho)|+|g_{2,a_3}^{(m)}(\rho)|
      &\stackrel{\eqref{eq:defg1a1g2a2g3a3}}{=}  2(2\sqrt t)^m\left|\frac{d^m}{dr^m}
      \left(\phi_\omega(r)B_n\left(\frac{r|x|}{2t}\right)\right) \right| 
      \left(\frac{|x|}{2t}\right)^{-\frac{n}{2}}
      \\
       &\les  (\sqrt t)^{m} \sum_{k=0}^{m} 
      |\phi_\omega^{(k)}(r)| \left|B_n^{(m-k)}\left(\frac{r|x|}{2t}\right)\right|
      \left(\frac{|x|}{2t}\right)^{m-k-\frac{n}{2}}   \\
      &\stackrel{\eqref{eq:Bnestimates}}\les (\sqrt t)^{m} \sum_{k=0}^{m} 
      |\phi_\omega^{(k)}(r)|
       \left(\frac{r|x|}{2t}\right)^{\frac{n-1}{2}-m +k}
        \left(\frac{|x|}{2t}\right)^{m-k-\frac{n}{2}}       
       \cdot   \ind_{[1/2,\infty)}\left(\frac{r|x|}{2t}\right) \\
     &\les |x|^{-\frac{1}{2}}(\sqrt t)^{m+1}  \sum_{k=0}^{m} 
      |\phi_\omega^{(k)}(r)|r^{\frac{n-1}{2}-m +k}
       \cdot   \ind_{[1/2,\infty)}\left(\frac{r|x|}{2t}\right).
    \end{align*} 
  \end{proof}
    
  \medskip
  
  \noindent\textbf{Proof of Theorem~\ref{thm:dispest}~(i)}: We combine the estimates from
  Proposition~\ref{prop:SolutionEstimate} and Proposition~\ref{prop:estimategja}. Under the assumption
  $\phi_\omega\in C_c^\infty(\R_{\geq 0};\C)$ we get for all $m\in \{0,\ldots,n-1\}$ 
    \begin{align*} 
      |\psi(x,t)|\; 
      &\hspace{-3mm}\stackrel{\text{Prop. }\ref{prop:SolutionEstimate}}{\les}  
      |x|^{\frac{2-n}{2}}(\sqrt t)^{-1}   \int_0^\infty |g_{1,a_1}^{(m)}(\rho)|
      +|g_{2,a_2}^{(m)}(\rho)|+|g_{3,a_3}^{(m)}(\rho)| \,d\rho \\
      &\hspace{-3mm}\stackrel{\text{Prop. }\ref{prop:estimategja}}{\les}  (\sqrt t)^{m-n+1}  \sum_{k=0}^m 
      \int_0^{\frac{\sqrt t}{|x|}} |\phi_\omega^{(k)}(2\sqrt t \rho)| (2\sqrt t \rho)^{n-m+k-1}  \,d\rho  \\
     &\qquad + |x|^{\frac{1-n}{2}}(\sqrt t)^{m}
    \sum_{k=0}^{m} 
    \int_{\frac{\sqrt t}{2|x|}}^\infty
       |\phi_\omega^{(k)}(2\sqrt t \rho)| (2\sqrt t \rho)^{\frac{n-1}{2}-m+k} \,d\rho \\ 
      &\les   (\sqrt t)^{m-n} \sum_{k=0}^m  \int_0^{\frac{2t}{|x|}}
    |\phi_\omega^{(k)}(r)|  r^{n-m+k-1} \,dr  \\
    & \qquad  + (\sqrt t)^{m-n}   \left(\frac{t}{|x|}\right)^{\frac{n-1}{2}}   
    \sum_{k=0}^{m} \int_{\frac{t}{|x|}}^\infty
      |\phi_\omega^{(k)}(r)|  r^{\frac{n-1}{2}-m+k} \,dr \\
    &\les   (\sqrt t)^{m-n} \sum_{k=0}^m  \left( \int_0^{\frac{2t}{|x|}}
    |\phi_\omega^{(k)}(r)|  r^{n-m+k-1} \,dr + \int_{\frac{t}{|x|}}^\infty
      |\phi_\omega^{(k)}(r)|  r^{n-m+k-1} \,dr \right) \\  
    &\les  (\sqrt t)^{m-n} \|\phi_\omega\|_{Y_m}. 
    \end{align*}
    In the case $m=n$ we use the second estimate in
    Proposition~\ref{prop:estimategja}  instead of the first one. By density of $C_c^\infty(\R_{\geq 0};\C)$ in $Y_m$
    the result follows.
    
    \medskip

  \noindent\textbf{Proof of Theorem~\ref{thm:dispest}~(ii)}: For $r=2\sqrt t \rho$ we use 
  $$
    |g_{1,a_1}'(\rho)|
    \les |x|^{\frac{n-2}{2}} (\sqrt t)^{3-n} \left( |\phi_\omega(r)| r^{n-2}
      + |\phi_\omega'(r)| r^{n-1} \right)   \cdot \ind_{[0,1]}\left(\frac{r|x|}{2t}\right) 
  $$    
  as well as    
  \begin{align*}
      &|g_{2,a_2}'(\rho)|+|g_{2,a_3}'(\rho)| \\
      &\stackrel{\eqref{eq:defg1a1g2a2g3a3}}{=}  4\sqrt t \left|\frac{d}{dr}
      \left(\phi_\omega(r)B_n\left(\frac{r|x|}{2t}\right)\right) \right| 
      \left(\frac{|x|}{2t}\right)^{-\frac{n}{2}}   \\
       &\les  \sqrt t  
      \left|\frac{d}{dr}\left( \phi_\omega(r)r^{\frac{n-1}{2}}\right)\right|   
      \left| r^{\frac{1-n}{2}} B_n\left(\frac{r|x|}{2t}\right) \right|
      \left(\frac{|x|}{2t}\right)^{-\frac{n}{2}}  \\
      &\quad+  \sqrt{t} |\phi_\omega(r)|r^{\frac{n-1}{2}} 
      \left|\frac{d}{dr}\left(r^{\frac{1-n}{2}} B_n\left(\frac{r|x|}{2t}\right)\right)\right| 
      \left(\frac{|x|}{2t}\right)^{-\frac{n}{2}}        \\
      &\hspace{-3mm}\stackrel{\text{Prop. }\ref{prop:AnBnProperties}}\les  
      \sqrt t
      \left|\frac{d}{dr}\left( \phi_\omega(r)r^{\frac{n-1}{2}}\right)\right| r^{\frac{1-n}{2}}
      \left(\frac{r|x|}{2t}\right)^{\frac{n-1}{2}} \left(\frac{|x|}{2t}\right)^{-\frac{n}{2}}
      \cdot \ind_{[1/2,\infty)}\left(\frac{r|x|}{2t}\right)\\
      &\quad + \sqrt t |\phi_\omega(r)|r^{\frac{n-1}{2}} 
        \left(\frac{r|x|}{2t}\right)^{-2}   
      \left(\frac{|x|}{2t}\right)^{\frac{1}{2}} \cdot \ind_{[1/2,\infty)}\left(\frac{r|x|}{2t}\right) \\
      & \les  \left( |x|^{-\frac{1}{2}}(\sqrt t)^2
      \left|\frac{d}{dr}\left( \phi_\omega(r)r^{\frac{n-1}{2}}\right)\right| 
      + |x|^{-\frac{3}{2}}(\sqrt t)^4 |\phi_\omega(r)|r^{\frac{n-5}{2}}\right)  
      \cdot \ind_{[1/2,\infty)}\left(\frac{r|x|}{2t}\right).
    \end{align*} 
  This implies 
  \begin{align*} 
      |\psi(x,t)|\; 
      &\hspace{-3mm}\stackrel{\text{Prop. }\ref{prop:SolutionEstimate}}{\les}  
      |x|^{\frac{2-n}{2}}(\sqrt t)^{-1}   \int_0^\infty |g_{1,a_1}^{(m)}(\rho)|
      +|g_{2,a_2}^{(m)}(\rho)|+|g_{3,a_3}^{(m)}(\rho)| \,d\rho \\ 
      &\hspace{-3mm}\stackrel{\text{Prop. }\ref{prop:estimategja}}{\les}  
      (\sqrt t)^{2-n}    \int_0^{\frac{\sqrt t}{|x|}} \left( |\phi_\omega(2\sqrt t \rho)|  (2\sqrt t
      \rho)^{n-2} + |\phi_\omega'(2\sqrt t \rho)|  (2\sqrt t \rho)^{n-1} \right) \,d\rho  \\
      &\quad +    |x|^{\frac{1-n}{2}}\sqrt t
       \int_{\frac{\sqrt t}{2|x|}}^\infty   
       \left|\frac{d}{dr}\left( \phi_\omega(r)r^{\frac{n-1}{2}}\right)\right|_{r=2\sqrt t \rho}  \,d\rho \\
      &\quad + |x|^{\frac{-1-n}{2}} (\sqrt t)^3  \int_{\frac{\sqrt t}{2|x|}}^\infty  |\phi_\omega(2\sqrt t
      \rho)|(2\sqrt t \rho)^{\frac{n-5}{2}} \,d\rho \\
     &\les   |x|^{\frac{1-n}{2}}   \left(\frac{t}{|x|}\right)^{\frac{1-n}{2}} \int_0^{\frac{2t}{|x|}}
     \left(|\phi_\omega(r)| r^{n-2} + |\phi_\omega'(r)|  r^{n-1} \right) \,dr  \\
     &\quad + |x|^{\frac{1-n}{2}}  \left(
       \int_{\frac{t}{|x|}}^\infty  
       \left|\frac{d}{dr}\left( \phi_\omega(r)r^{\frac{n-1}{2}}\right)\right| \,dr
       + \frac{t}{|x|}
        \int_{\frac{t}{|x|}}^\infty |\phi_\omega(r)|r^{\frac{n-5}{2}} \,dr
       \right)\\
     &\les  |x|^{\frac{1-n}{2}} \|\phi_\omega\|_X.   
  \end{align*} 
   So we get the result by density of $C_c^\infty(\R_{\geq 0};\C)$ in $X$. 
 \qed

\section{Proof of Theorem~\ref{thm:NoStrichartz}}

We estimate the solution of the Schr\"odinger equation for the initial datum  
$$
  \phi(x)=\phi_{\rad}(|x|)\quad\text{where } \phi_{\rad}(\rho):= e^{-i\frac{\rho^2}{4}} \ind_{\rho\geq 1}
  \rho^{-\sigma} 
$$ 
and $\sigma$ is chosen according to $\frac{n-3}{2}<\sigma<n$. In this case, the formula~\eqref{eq:SolutionFormulaI} is well-defined and provides a
solution of the initial value problem~\eqref{eq:SchroedingerFlow}.
In~\cite[Section~2.1]{BoPoSaSp_DispersiveBlowup} it was shown that the solution $\psi$ blows up in $L^\infty(\R^n)$ as $t\to 1$
provided that $\frac{n}{2}<\sigma<n$ holds. In fact, in this case the
function $a(x):= \ind_{|x|\geq 1} |x|^{-\sigma}$ lies in
$L^2(\R^n)$ but not in $L^1(\R^n)$ so that \cite[Remark~2.2]{BoPoSaSp_DispersiveBlowup} applies. 
We now generalize this analysis to the range $\frac{n-3}{2}<\sigma<n$ and detect a selfsimilar blow-up 
in $L^q(\R^n)$ for all $q>\frac{n}{n-\sigma}$ and a lower estimate for the corresponding blow-up
rate then implies $\|\psi\|_{L^p(\R,L^q(\R^n))}=\infty$ for $p\geq \frac{2q}{((n-\sigma)q-n)_+}$. From
this we will finally deduce the nonvalidity of Strichartz estimates for initial data $\phi\in L^r(\R^n)$ where
$r>2$.

\medskip

We set $k_t:= \sqrt{\frac{1}{4t}-\frac{1}{4}}$ for $0\leq t<1$ and write
$\psi(x)=\psi_{\rad}(|x|)$. We get for $|x|=2tk_t z$  
\begin{align*}
  2|\psi(x,t)| k_t^{n-\sigma}
  &\stackrel{\eqref{eq:SolutionFormulaI}}=   
  |x|^{\frac{2-n}{2}} t^{-1} k_t^{n-\sigma} \left| \int_0^\infty J_{\frac{n-2}{2}}\left(\frac{\rho
  |x|}{2t}\right)\phi_{\rad}(\rho)\rho^{\frac{n}{2}}e^{i\frac{\rho^2}{4t}}\,d\rho \right| \\
  &= |x|^{\frac{2-n}{2}}t^{-1} k_t^{n-\sigma} \left| \int_1^\infty J_{\frac{n-2}{2}}\left(\frac{\rho
  |x|}{2t}\right)\rho^{\frac{n}{2}-\sigma}e^{i\rho^2 k_t^2}\,d\rho \right| \\
  &= (2tk_t z)^{\frac{2-n}{2}}t^{-1} k_t^{n-\sigma}  \left| \int_1^\infty J_{\frac{n-2}{2}}( \rho
  k_t z)\rho^{\frac{n}{2}-\sigma}e^{i\rho^2 k_t^2}\,d\rho \right| \\
  &= (2tk_t z)^{\frac{2-n}{2}}t^{-1} k_t^{\frac{n-2}{2}} 
  \left| \int_{k_t}^\infty J_{\frac{n-2}{2}}(sz)s^{\frac{n}{2}-\sigma}e^{is^2}\,ds \right| \\
  &\to  (2z)^{\frac{2-n}{2}} 
  \left| \int_0^\infty J_{\frac{n-2}{2}}(sz)s^{\frac{n}{2}-\sigma}e^{is^2}\,ds \right|
  \qquad\text{as }t\to 1 
\end{align*}
and the convergence is locally uniform in $z\in (0,\infty)$ due to  $\frac{n-3}{2}<\sigma<n$. 
Since the right hand side is not identically zero we may find $\delta>0$ and
radii $0<R_1<R_2$ such that 
$$
  |\psi(x,t)| 
  \gtrsim\, k_t^{\sigma-n} \qquad\text{for } R_1k_t \leq |x|  \leq R_2k_t \;\text{ and
  }\; 1-\delta<t<1.
$$ 
Hence we get 
\begin{align*}
  \int_{1-\delta}^1 \left(\int_{B_{R_2k_t}(0)\sm B_{R_1k_t}(0)} |\psi(x,t)|^q\,dx\right)^{p/q}\,dt
  &\gtrsim \int_{1-\delta}^1 \left(\int_{R_1k_t}^{R_2k_t} r^{n-1}  k_t^{(\sigma-n)q}\,dr \right)^{p/q}\,dt
  \\
  &\gtrsim   \int_{1-\delta}^1 \left( k_t^{n+(\sigma-n)q}  \right)^{p/q}\,dt \\
  &\gtrsim \int_{1-\delta}^1 (1-t)^{\frac{p}{2q}(n+(\sigma-n)q)} \,dt.  
\end{align*}  
This integral is finite if and only if $\frac{p}{2q}(n+(\sigma-n)q) >-1$. 
So $\psi\in L^p(\R;L^q(\R^n))$ can only hold for $p < \frac{2q}{((n-\sigma)q-n)_+}$. 
Moreover, the initial datum lies in $L^r(\R^n)$ if and only if $\sigma>\frac{n}{r}$. So, 
for any $r>1$ we can consider the limit $\sigma\searrow \max\{\frac{n}{r},\frac{n-3}{2}\}$ 
we find that the validity of the Strichartz estimate \label{eq:Strichartz_est} with initial datum in
$L^r(\R^n), r>1$ implies
\begin{equation} \label{eq:bound_for_p}
  p 
  \leq \frac{2q}{((n-\max\{\frac{n}{r},\frac{n-3}{2}\})q-n)_+}
  = \max\left\{\frac{2qr}{n((r-1)q-r)_+},\frac{4q}{((n+3)q-2n)_+}\right\}. 
\end{equation}
On the other hand, the scaling invariance of the Schr\"odinger equation implies
$\frac{2}{p}+\frac{n}{q}=\frac{n}{r}$ and thus $p=\frac{2qr}{n(q-r)}$. Plugging this
into~\eqref{eq:bound_for_p} we obtain $r\leq 2$. Hence, the Strichartz
estimate~\eqref{eq:Strichartz_est} cannot hold for any $r>2$, which finishes the proof. \qed

\section{Appendix}
    
  In this Appendix we briefly discuss the restriction $p\geq 2$ in the context of Strichartz estimates of the
  form 
  $$
     \|e^{it\Delta} \phi\|_{L^p_t(\R;L^q(\R^n))} \les \|\phi\|_{L^2(\R^n)},
  $$
  which results from an abstract reasoning involving translation-invariant operators due to
  H\"ormander~\cite{Hoermander}, see~\cite[p.970-971]{KeelTao}. Here we provide a family of explicit
  counterexamples that not only implies $p\geq 2$ for square integrable initial data, but even shows $p\geq
  \frac{2r}{(2n-r(n-1))_+}$ for initial data in $L^r(\R^n)$ with $r>\frac{2n}{n+1}$. In order to avoid
  lengthy computations involving oscillatory integrals, we only sketch the proofs.
  The starting point is a reasonable choice of an initial condition. We choose
  \begin{equation*}
    \phi(x)  = |x|^{\frac{2-n}{2}} \int_1^2 (\omega-1)^{-\delta} J_{\frac{n-2}{2}}(\omega |x|)\,d\omega  
  \end{equation*}
  for   $\delta\in (0,1)$. This function  corresponds to a singular superposition of Herglotz
  waves, cf. Remark~\ref{rem}~(c). It is smooth and lenghty computations involving the van der Corput Lemma
  \cite[p.334]{Stein_Harmonic} reveal 
  $$
    |\phi(x)|
    \sim 2|x|^{\frac{1-n}{2}-(1-\delta)}   \int_0^\infty \Real\left(e^{i\rho} \alpha_0 
    e^{i(|x|-\frac{(n-1)\pi}{4})} \right) \rho^{-\delta}\,d\rho
    \qquad\text{as }|x|\to\infty
  $$
  where $\alpha_0>0$ is the domininant term in the series expansion of the Bessel funtion near infinity,
  see~\eqref{eq:def_alphak}. In particular we get $\phi\in L^r(\R^n)$ if and only if
  $\delta<\frac{n+1}{2}-\frac{n}{r}$. 
  
  \medskip
  
  The above choice for the initial datum allows to write down the corresponding solution of the
  Schr\"odinger equation semi-explicitly via 
  $$
    \widehat{\psi(\cdot,t)}(\xi)
    = e^{-it|\xi|^2}\hat \phi(\xi)
    =  e^{-it|\xi|^2} |\xi|^{-\frac{n}{2}} (|\xi|-1)^{-\delta} \ind_{[1,2]}(|\xi|).
  $$ 
  Hence, one gets
  \begin{align*}
    \psi(x,t) 
     &= |x|^{\frac{2-n}{2}} 
     \int_1^2 e^{-it\omega^2} (\omega-1)^{-\delta} J_{\frac{n-2}{2}}(\omega|x|)  \,d\omega
  \end{align*}
 and the van der Corput Lemma implies
  $|\psi(x,t)| \gtrsim c t^{\delta-1}$  for small $|x|$ and large $t$. In particular, 
  $\psi\in L^p(\R;L^q(\R^n))$ implies $p(1-\delta)>1$. So we conclude that for any $r\in
  (\frac{2n}{n+1},\infty]$ we can consider the limit $\delta\nearrow \min\{1,\frac{n+1}{2}-\frac{n}{r}\}$ in
  the above computations and obtain that $\psi\in L^p(\R;L^q(\R^n))$ implies 
  \begin{align}\label{eq:p_lower_bound} 
      p \geq \frac{2r}{(2n-r(n-1))_+}\qquad\text{provided }r>\frac{2n}{n+1}.
  \end{align} 
  For $r=2$, i.e., for square integrable initial data, this implies $p\geq 2$, which is all we wanted
  to demonstrate.
  
  \medskip
  
  Let us mention that a detailed analysis of $\psi$ reveals $\psi\in L^p(\R;L^q(\R^n))$ if and only if 
  \begin{align*}
      \begin{aligned}
      q &> \max\left\{ \frac{2n-1}{n-\delta}, \frac{2n}{n+1-2\delta}\right\} \qquad\text{and} \\
      p &>\max\left\{\frac{1}{1-\delta},\frac{2q}{q(n-\delta)+1-2n}, \frac{2q}{q(n+1-2\delta)-2n}\right\}.
    \end{aligned}
  \end{align*}
  Keeping the scaling condition $\frac{2}{p}+\frac{n}{q}=\frac{n}{r}$ in mind, these a priori more restrictive
  conditions do however not result in stronger necessary conditions than~\eqref{eq:p_lower_bound}, so that further necessary
  conditions cannot be deduced from this example. 
   
\section*{Acknowledgements}

Funded by the Deutsche Forschungsgemeinschaft (DFG, German Research Foundation)
- Project-ID 258734477 - SFB 1173. The author thanks Dominic Scheider (Karlsruhe Institute of
Technology) and Martin Spitz (University of Bielefeld) for valuable remarks and discussions leading to an
improvement of the manuscript.

\bibliographystyle{abbrv}	
\bibliography{doc}

 \end{document}